\documentclass[a4paper, article, oneside, UKenglish, final]{memoir}

\usepackage[utf8]{inputenx} 
\usepackage[T1]{fontenc}    

\usepackage{lmodern}           
\usepackage[scaled]{beramono}  
\usepackage[final]{microtype}  
\pretitle{\begin{center}\LARGE\sffamily\bfseries}                    
\setsecheadstyle{\large\bfseries\sffamily\raggedright}               
\setsubsecheadstyle{\large\bfseries\sffamily\raggedright}            
\setsubsubsecheadstyle{\normalsize\bfseries\sffamily\raggedright}    
\setparaheadstyle{\normalsize\bfseries\sffamily\raggedright}         
\setsubparaheadstyle{\normalsize\bfseries\sffamily\raggedright}      

\usepackage{amssymb}   
\usepackage{amsthm}    
\usepackage{thmtools}  
\usepackage{mathtools} 
\usepackage{mathrsfs}  
\numberwithin{equation}{chapter}

\RequirePackage{enumitem}
\setlist[itemize]{font = \upshape, before = \leavevmode}
\setlist[enumerate]{font = \upshape, before = \leavevmode}
\setlist[description]{font = \bfseries\sffamily, before = \leavevmode}

\usepackage{xspace}    
\usepackage{graphicx}  
\usepackage{babel}     
\usepackage{csquotes}  
\usepackage{textcomp}  
\usepackage{listings}  
\usepackage[
            textwidth = 40mm,
            textsize = footnotesize]{todonotes}
\usepackage[notref, notcite]{showkeys}

\RequirePackage[backend     = biber,
                sortcites   = true,
                giveninits  = true,
                maxbibnames = 5,
                doi         = false,
                isbn        = false,
                url         = false,
                style       = alphabetic]{biblatex}
\DeclareNameAlias{sortname}{family-given}
\DeclareNameAlias{default}{family-given}
\DeclareFieldFormat[article]{volume}{\bibstring{jourvol}\addnbspace#1}
\DeclareFieldFormat[article]{number}{\bibstring{number}\addnbspace#1}
\renewbibmacro*{volume+number+eid}
{
    \printfield{volume}
    \setunit{\addcomma\space}
    \printfield{number}
    \setunit{\addcomma\space}
    \printfield{eid}
}

\usepackage{varioref}
\usepackage[colorlinks, allcolors = uiolink]{hyperref}
\usepackage[nameinlink, capitalize, noabbrev]{cleveref}

\DeclareSymbolFont{extraup}{U}{zavm}{m}{n}
\DeclareMathSymbol{\vardiamond}{\mathalpha}{extraup}{87}

\declaretheoremstyle[headfont   = \bfseries\sffamily,
                     notefont   = \normalfont,
                     bodyfont   = \itshape,
                     spaceabove = 6pt,
                     spacebelow = 6pt]{plain}
\declaretheoremstyle[headfont   = \bfseries\sffamily,
                     notefont   = \normalfont,
                     spaceabove = 6pt,
                     spacebelow = 6pt]{definition}
\declaretheorem[style = plain, numberwithin = chapter]{theorem}
\declaretheorem[style = plain,      sibling = theorem]{corollary}
\declaretheorem[style = plain,      sibling = theorem]{lemma}
\declaretheorem[style = plain,      sibling = theorem]{proposition}
\declaretheorem[style = definition, sibling = theorem]{definition}
\declaretheorem[style = definition, sibling = theorem, qed = \(\vardiamond\)]{example}
\declaretheorem[style = definition, sibling = theorem, qed = \(\spadesuit\)]{remark}

\DeclarePairedDelimiter{\set}{\lbrace}{\rbrace} 
\DeclarePairedDelimiter{\ip}{\langle}{\rangle}

\DeclareMathOperator{\rank}{rank}
\DeclareMathOperator{\corank}{corank}
\DeclareMathOperator{\Sec}{Sec}
\DeclareMathOperator{\Sing}{Sing}
\DeclareMathOperator{\Bl}{Bl}

\newcommand{\R}{\mathbb{R}}   
\newcommand{\C}{\mathbb{C}}   
\renewcommand{\P}{\mathbb{P}} 
\newcommand{\V}{\mathcal{V}\mkern-1mu}

\newcommand{\y}{\mathbf{y}}

\newcommand{\point}[1]{rank-\(#1\)-point}
\newcommand{\locus}[1]{rank-\(#1\)-locus}

\newcommand{\node}[1]{rank-\(#1\)-node}
\newcommand{\quadric}[1]{rank-\(#1\)-quadric}
\newcommand{\baselocus}{base locus\xspace}
\newcommand{\baseloci}{base loci\xspace}
\newcommand{\basepoint}{base point\xspace}
\newcommand{\basepoints}{base points\xspace}

\newcommand{\MXkern}{\ensuremath{\mkern-7mu}}
\newcommand{\Pnkern}{\ensuremath{\mkern-4mu}}
\newcommand{\Pkern}{\ensuremath{\mkern-3mu}}
\newcommand{\Vkern}{\ensuremath{\mkern-4mu}}
\newcommand{\Wkern}{\ensuremath{\mkern-4mu}}
\newcommand{\Skern}{\ensuremath{\mkern-1mu}}

\newcommand{\Ckern}{\ensuremath{\mkern-2mu}}
\newcommand{\Hkern}{\ensuremath{\mkern-2mu}}
\newcommand{\Xkern}{\ensuremath{\mkern-2mu}}
\newcommand{\Ykern}{\ensuremath{\mkern-4mu}}

\newcommand{\dash}{\textthreequartersemdash\xspace}
\newcommand{\bezout}{B{\'e}zout's theorem\xspace}
\definecolor{uiolink}{HTML}{0B5A9D}

\title{Maximality of quartic symmetroids\par with a double quadric of codimension 1}
\author{Martin Helsø}
\date{}

\begin{document}

\maketitle

\begin{abstract}
    \noindent
    We prove that the dimension of a quartic symmetroid singular along a quadric of codimension 1 is at most 4,
    if it is not a cone.
    In the maximal case, the quadric is reducible and consists of rank-3-points.
    If the quadric is irreducible,
    it consists of rank-2-points
    and the symmetroid is at most 3-dimensional.
\end{abstract}

\chapter{Introduction}

\noindent
In semidefinite programming,
the objective is to optimise a linear function
over the intersection of the cone of positive semidefinite matrices
in the space of real, symmetric $(d \times d)$-matrices with an affine subspace.
The feasible region of a semidefinite program is called a \emph{spectrahedron}.
Spectrahedra are important, elementary objects in convex algebraic geometry \cite{BPT13}.

Quartic spectrahedra is the case of $(4 \times 4)$-matrices intersected with an affine space
that contains a positive definite matrix.
The affine space is identified with $\R^n$\Pnkern.
The algebraic boundary of a spectrahedron in $\R^n$ is a hypersurface~$\V(f)$ in $\C\P^n$\Pnkern. 
The polynomial~$f$ is the determinant of a matrix
\begin{equation}
    \label{eq:matrix-rep}
    A(x) \coloneqq A_0x_0 + A_1x_1 + \cdots + A_nx_n,
\end{equation}
where $A_i$ is a real, symmetric $(4 \times 4)$-matrix.
Classically, the hypersurface~$\V(f)$ is called a \emph{symmetroid} \cite{Cay69, Dol12, Jes16}.
Given a matrix~$A(x)$ as in \eqref{eq:matrix-rep},
with $f \coloneqq \det(A(x))$,
the \emph{associated spectrahedron} is the set
\begin{equation*}
    S(f) \coloneqq \set{x \in \R\P^n \mid A(x) \text{ is semidefinite}}.
\end{equation*}
We say the symmetroid~$\V(f)$ is \emph{spectrahedral} or has a \emph{nonempty spectrahedron}
if $S(f)$ contains a positive definite matrix.

For a given dimension~$n$,
the generic quartic symmetroid $\V(f) \subset \C\P^n$ has a fixed number of singularities.
Some of the singularities may be real,
and some of the real singularities may lie on the topological boundary of the spectrahedron.
For generic quartic symmetroids in $\C\P^3$ with a nonempty spectrahedron,
the possible arrangements of singularities are characterised in \cite{DI11}.
This result is recovered in \cite{Ott+14}, using an algorithmic proof.
The singular loci of rational quartic symmetroids in $\C\P^3$ are described in \cite{Hel17}.
The possible locations of the singularities of rational quartic symmetroids with a nonempty spectrahedron are further specified by \cite{HR18}.

Rational quartic symmetroids in $\C\P^3$ include quartic surfaces with a double conic section.
In this paper, we consider the analogous situation in higher dimensions.
Our main result is:

\begin{theorem}
    Let $S \subset \P^n$ be an irreducible, quartic, symmetroidal hypersurface
    which is not a cone.
    Assume that $Q$ is an $(n - 2)$-dimensional, quadratic component
    in $\Sing(S)$.
    Then $n \leqslant 4$ if $Q$ is irreducible and $n \leqslant 5$ otherwise.
\end{theorem}

\noindent
This is proved by combining \cref{cor:max-smooth,prop:max-union}.
The bounds are only valid for symmetroids which are not cones,
as \eqref{eq:matrix-rep} defines a symmetroid in $\C\P^{m}$ for any $m \geqslant n$.

The rest of the paper is organised as follows:
\cref{sec:prelim} introduces notation
and basic facts about symmetroids and linear systems of quadrics.
\cref{sec:max-irred} demonstrates that if $S$ is a symmetroid with an irreducible, double quadric~$Q$ of codimension~$1$,
then points in $Q$ have at most rank~$2$ and $\dim(Q) \leqslant 2$.
\cref{sec:threefolds} deals with the case where $S$ is a threefold.
First we find the additional singularities outside of $Q$.
Then we describe the different configurations of singularities when $S$ is spectrahedral.
In \cref{sec:threefolds-cyclides},
we argue that $Q$ must contain real points
if $S$ is a threefold.
\cref{sec:max-red} shows that if we allow $Q$ to be reducible with \point{3}s,
then $\dim(Q) \leqslant 3$.
\cref{sec:fourfolds} determines the number of additional singularities
outside of $Q$ when $S$ is a fourfold,
and examines them further in the spectrahedral case.

\chapter{Preliminaries}
\label[section]{sec:prelim}

Here we present our main tool,
the linear system of quadrics associated to a symmetroid,
and sundry useful results.

First,
note that throughout the paper,
we assume that the symmetroids under consideration are not cones.

Our notation is as follows:
Let $S \subset \C\P^n$ be a quartic symmetroid with representation~\eqref{eq:matrix-rep}.
For $x \in \C\P^n$ and $\y \coloneqq [y_0, y_1, y_2, y_3]$,
let $q(x) \coloneqq \y^T\mkern-7mu A(x) \y$.
Then $Q(x) \coloneqq \V(q(x)) \subset \C\P^{3}$ is a quadratic surface.
The set
\begin{equation*}
W(S) \coloneqq \big\{ Q(x) \mid x \in \P^n \big\}
\end{equation*}
is the \emph{associated linear system of quadrics} of $S$\Skern.
For a subset $U \subset S$,
we let
\begin{equation*}
W(U) \coloneqq \big\{ Q(x) \mid x \in \overline{U} \big\} \subset W(S).
\end{equation*}
For a quadric $Q \subset \C\P^3$\Pnkern, let $[Q]$ denote the corresponding point in $W(S) = \C\P^n$\Pnkern.
For a space~$W$ of quadrics, the \baselocus is written as $\Bl(W)$.
For a subset $\Bl(W) \subset \P^3$\Pnkern,
we let $X_{\Bl(W)} \subset \P^9$ denote the space of all quadrics containing $\Bl(W)$.

The \emph{rank} and \emph{corank} of a point $x \in \C\P^n$ are defined as $\rank A(x)$ and $\corank A(x)$, respectively.
The \emph{\locus{k}} of~$S$ \dash or of $W(S)$ \dash is the set of points with rank less than or equal to $k$.
If $S$ is a symmetroid of degree~$d$, then the \locus{(d - 2)} is contained in $\Sing(S)$,
but equality does not always hold.
The following well-known connection between the \baselocus of $W(S)$ and the singular \point{(d - 1)}s can be found in \cite[Lemma~1.1]{Wal81} and \cite[Lemma~2.13]{Ili+17}:

\begin{lemma}
    \label{lem:basepoint-singularity}
    Let $S \subset \P^n$ be a symmetroid.
    If $[Q] \in \P^n$ is a point such that $Q$ is a singular quadric with a singularity at a point $p \in \Bl(W(S))$,
    then $[Q] \in \Sing(S)$.
    If $[Q] \in \Sing(S)$ is a co\point{1} and $p$ is the singular point of $Q$,
    then $p \in \Bl(W(S))$.
\end{lemma}

\noindent
The following lemma is useful for eliminating possible symmetroids:

\begin{lemma}[{\cite[Lemma~2.7]{Hel17}}]
    \label{lem:base-curve}
    Let $S \subset \P^n$ be a quartic symmetroid and assume that $\Bl(W(S))$ contains a curve.
    Then $S$ is reducible.
\end{lemma}

\noindent
The next result elaborates on the relationship between singular co\point{1}s and $\Bl(W(S))$ for quartics:

\begin{lemma}
    \label{lem:common-singularity}
    Let $S \subset \P^n$ be an irreducible quartic symmetroid
    and suppose that $H \subseteq \Sing(S)$
    is a linear space of \point{3}s.
    Then the quadrics in $W(H)$ share a common singularity,
    which is a \basepoint for $W(S)$.
\end{lemma}

\begin{proof}
    By \cref{lem:basepoint-singularity},
    the singularities of the quadrics in $W(H)$ are \basepoints for $W(S)$.
    If the quadrics in $W(H)$ have different singularities,
    then $\Bl(W(S))$ contains a curve.
    This is impossible by \cref{lem:base-curve}.
\end{proof}

\noindent
We also need some facts about special linear systems of quadrics,
beginning with pencils:

\begin{lemma}
    \label{lem:rank-2-pencil}
    Let $P$ be a pencil of quadrics in $\P^n$ and assume that the generic member is a \quadric{2}.
    Then one of the following is true:
    \begin{enumerate}
        \item
        There are no \quadric{1}s in $P$\Pkern,
        and $\Bl(P)$ consists of a hyperplane~$H$ and a linear subspace~$L \not\subset H$ of codimension~$2$;
        
        \item
        There is one \quadric{1} in $P$\Pkern,
        and $\Bl(P)$ consists of a hyperplane~$H$ and a double linear subspace~$L \subset H$ of codimension~$2$;
        
        \item
        There are two \quadric{1}s in $P$\Pkern,
        and $\Bl(P)$ contains a double linear subspace~$L$ of codimension~$2$.
    \end{enumerate}
\end{lemma}

\begin{proof}
    The first two statements can be read off of the proof of \cite[Lemma~2.5]{Hel17}.
    The last statement follows by considering the two \quadric{1}s as generators for $P$\Pkern.
\end{proof}

\noindent
Using \cref{lem:rank-2-pencil} as the base case for induction,
we generalise the statement to linear systems of any dimension:

\begin{lemma}
    \label{lem:hyperplane-in-base}
    Let $W$ be a linear system of quadrics in $\P^n$ and assume that the generic member is a \quadric{2}.
    Then one of the following is true:
    \begin{enumerate}
        \item
        There are no \quadric{1}s in $W$\Wkern,
        and $\Bl(W)$ contains a hyperplane;
        
        \item
        There is one \quadric{1} in $W$\Wkern,
        and $\Bl(W)$ contains a hyperplane;
        
        \item
        There is a quadratic hypersurface of \point{1}s in $W$\Wkern,
        and $\Bl(W)$ contains a double linear subspace~$L$ of codimension~$2$.
    \end{enumerate}
\end{lemma}

\begin{proof}
    We prove this by induction on $\dim(W)$.
    The base case $\dim(W) = 1$ is true by \cref{lem:rank-2-pencil}.
    Let $k \coloneqq \dim(W)$
    and assume that the assertions hold for linear systems of dimension less than or equal to $k - 1$.
    \begin{enumerate}
        \item
        If there are no \quadric{1}s in $W$\Wkern,
        then $\Bl(W')$ contains a hyperplane
        for all proper, nonempty linear subsets $W' \subset W$\Wkern.
        Let $W_1$, $W_2$, $W_3 \subset W$ be three linear subsets of codimension~$1$ containing the quadric $H_1 \cup H_2$.
        Then one of the hyperplanes~$H_1$,~$H_2$ is contained in at least two of the loci $\Bl(W_1)$, $\Bl(W_2)$ and $\Bl(W_3)$.
        That hyperplane is contained in $\Bl(W)$.
        
        \item
        If $2H$ is the only \quadric{1} in $W$\Wkern,
        then $H \subset \Bl(W')$
        for all proper, linear subsets $W' \subset W$ containing $2H$\Hkern.
        Thus $H \subset \Bl(W)$.
        
        \item
        Assume that there is more than one \quadric{1} in $W$
        and let $2H$ be one of them.
        Since there is no hyperplane contained in $\Bl(W)$,
        by the previous case,
        there is at most one hyperplane $W' \subset W$ containing $2H$ and no other \quadric{1}s.
        Hence $W$ contains a hypersurface~$S$ of \quadric{1}s.
        A generic hyperplane~$W' \subset W$ intersects $S$ in a quadric,
        so $S$ has degree~$2$.
        
        Restricting $W$ to $H$ defines a linear system~$W''$ of quadrics in $\P^{n - 1}$\Pnkern.
        Each line in $W$ through $[2H]$ is collapsed to a point in $W''$\Wkern.
        For each point in $W''$\Wkern, we can find a representative in $S$\Skern.
        It follows that $W''$ is a linear system of double hyperplanes in $\P^{n - 1}$\Pnkern.
        Thus $W''$ consists of a single double hyperplane~$L \subset \P^{n - 1}$\Pnkern,
        which is contained in $\Bl(W)$. 
        \qedhere
    \end{enumerate}
\end{proof}

\begin{lemma}
    \label{lem:codim-2-singular-subspace}
    Let $W$ be a linear system of quadrics in $\P^3$ with \basepoint~$p$.
    Let $V \subset W$ be the subspace of quadrics singular at a point $p$.
    Suppose that $\dim(V) = \dim(W) - 2$.
    Then $\Bl(W)$ contains a scheme of length~$2$ with support in $p$.
\end{lemma}

\begin{proof}
    Let $Q_1$, $Q_2$ be two quadrics that generate $W$ together with $V$\Vkern.
    Let $P_i$ be a pencil generated by $Q_i$ and a quadric~$Q'_i$ in $V$\Vkern.
    Since $p$ is a \basepoint for $P_i$ and $Q'_i$ is singular at $p$,
    the \baselocus~$\Bl(P_i)$ is a quartic curve~$C_i$ with a singularity at $p$.
    Let $T_i$ be the tangent plane of $C_i$ at $p$.
    It follows that $T_i$ is a common tangent plane
    to all the quadrics in $P_i$.
    Hence all quadrics in $P_i$ contain the first order infinitesimal neighbourhood of $p$ in $T_i$.
    Since all quadrics in $V$ are singular at $p$,
    we get that $\Bl(W)$ contains the first order infinitesimal neighbourhood
    of $p$ in the line $T_1 \cap T_2$.
\end{proof}

\begin{remark}
    \label{rmk:four-basepoints}
    Repeatedly in this paper,
    we encounter a symmetroid~$S$ where the \baselocus~$\Bl(W(S))$ contains four points. 
    Therefore,
    we make a general remark here about the space~$X_{\Bl}$
    of all quadrics containing four \basepoints
    $\Bl \coloneqq \{p_1, p_2, p_3, p_4\}$.
    We describe the \locus{2} of $X_{\Bl}$ in the two cases
    where $\Bl$ consists of four general or four coplanar points.
    
    In both cases,
    there are three quadratic surfaces,
    $X_{12}$, $X_{13}$, $X_{14}$ in the \locus{2} of $X_{\Bl}$.
    Indeed, let $H_{ij}$ and $H_{kl}$ be two planes
    containing the lines $\overline{p_i, p_j}$ and $\overline{p_k, p_l}$, respectively.
    The set of pairs $H_{ij} \cup H_{kl}$
    forms a smooth quadratic surface $X_{ij} \subset X_{\Bl}$ of \point{2}s.
    
    If $\Bl$ spans $\P^3$\Pnkern,
    let $H_i$ be any plane passing through $p_i$.
    The set $X_i$ of all unions~$H_i \cup \overline{p_j, p_k, p_l}$ forms a plane in $X_{\Bl}$.
    On the other hand,
    if the \basepoints are coplanar,
    the \locus{2} contains a linear $3$-space~$X$
    instead of the planes $X_1$, $X_2$, $X_3$, $X_4$.
    The quadrics in $X$ are the unions of any plane in $\P^3$
    with  the plane~$\overline{p_1, p_2, p_3, p_4}$.
\end{remark}

\chapter{Maximality of irreducible quadrics}
\label[section]{sec:max-irred}

In this section,
we prove that
the dimension of an irreducible quadric
that can be a component in the singular locus of a quartic symmetroid
is bounded by $2$.
Hence a quartic symmetroid with a double irreducible quadric of codimension~$1$
is at most $3$-dimensional.

We can give the first characterisation of singular quadrics of codimension~$1$
by modifying the proof of \cite[Proposition~4.1]{Hel17} slightly to be valid in any dimension:

\begin{lemma}
    \label{lem:quadric-in-rank-2}
    Let $S \subset \P^n$ be an irreducible, quartic symmetroid that is double along an irreducible, $(n - 2)$-dimensional quadric~$Q$.
    Then $Q$ is contained in the \locus{2} of $S$\Skern.
\end{lemma}

\begin{proof}
    Assume for contradiction that $Q$ is not contained in the \locus{2} of~$S$\Skern.
    A generic point in $Q$ is then a \node{3}.
    Let $[Q_0], \ldots, [Q_{n - 2}] \in Q$ be \point{3}s that span a linear subspace $H \subset \P^n$ of codimension~$2$.
    Suppose that $Q_0, \ldots, Q_{n - 2}$ share a common singular point~$p$.
    Then all the quadrics in~$H$ are singular at $p$.
    By \cref{lem:basepoint-singularity},
    this implies that $H$ is contained in $\Sing(S)$.
    Since $Q$ is irreducible,
    $H$ is not a component of $Q$.
    The space spanned by $Q$ intersects $S$ in $Q$ and $H$\Hkern,
    and it follows that it is contained in~$S$\Skern.
    This is impossible since $S$ is irreducible.
    We conclude that $Q_0, \ldots, Q_{n - 2}$ do not share a common singularity.
    Then there exists a curve of singularities of \node{3}s in $Q$.
    \cref{lem:basepoint-singularity,lem:base-curve} imply that $S$ is reducible,
    which again is impossible.
    Hence $Q$ is contained in the \locus{2} of $S$\Skern.
\end{proof}

\noindent
\cref{lem:quadric-in-rank-2} leads us to find the maximal dimension of an irreducible quadratic variety consisting of \point{2}s.
We know \emph{a priori} that the \quadric{2}s form a $6$-dimensional variety in the $\P^9$ of all quadrics in $\P^3$\Pnkern,
but requiring that the quadratic variety should be a component in the \locus{2} lowers the possible dimension:

\begin{proposition}
    \label{prop:max-sing-quadric}
    Let $S \subset \P^n$ be a quartic symmetroid and assume that $Q$ is an $(n - 2)$-dimensional, singular quadratic component in the \locus{2} of $S$\Skern.
    Assume that $Q$ contains at most finitely many \point{1}s.
    Then $n \leqslant 3$.
\end{proposition}

\begin{proof}
    It suffices to show that $n \neq 4$.
    Assume for contradiction that $n = 4$.
    We proceed by eliminating case by case:
    \begin{enumerate}
        \item
        Assume that $Q$ is reducible.
        Then $Q$ contains a plane~$H\Hkern$.
        Because $Q$ contains finitely many \point{1}s,
        it follows from \cref{lem:hyperplane-in-base} that $\Bl(W(H))$ is a plane.
        Since $W(S)$ equals the span of $W(H)$ and two quadrics,
        we have that $\Bl(W(S))$ contains four coplanar points.
        In the notation of \cref{rmk:four-basepoints},
        $W(S)$ intersects $X$ in $Q$.
        \bezout implies that the $3$-space $X$ is contained in the \locus{2} of $S$,
        so $Q$ is not a component in the \locus{2}.

        \item
        Assume that $Q$ is an irreducible cone.
        Let $p \coloneqq \Sing(Q)$ be the apex of~$Q$.
        Let $H_1 \cup H_2 \subset \P^3$ be the quadric corresponding to $p$.
        Since $Q$ contains at most finitely many \point{1}s,
        there are at least three lines through $p$ that correspond to pencils of type~$1$ or $2$ in the numbering of \cref{lem:rank-2-pencil}.
        The \baseloci of each of these pencils contain either $H_1$ or $H_2$.
        Hence there are two pencils having a plane in common, say $H_1$.
        Then $H_1 \subseteq \Bl(N)$,
        where $N$ is the net spanned by these two pencils.
        The elements of $N$ are \quadric{2}s,
        so $N$ corresponds to a plane contained in $Q$,
        contradicting the irreducibility.
        \qedhere
    \end{enumerate}    
\end{proof}

\begin{corollary}
    \label{cor:max-smooth}
    Let $S \subset \P^n$ be a quartic symmetroid and assume that $Q$ is an $(n - 2)$-dimensional, smooth quadratic component in the \locus{2} of $S$\Skern.
    Then $n \leqslant 4$.
\end{corollary}

\begin{proof}
    Assume for contradiction that $n = 5$,
    so $\dim(Q) = 3$.
    Let $p \in Q$ be a point,
    and $L$ a linear subspace that strictly contains the tangent space $T_pQ$, but not $Q$.
    Let $Q'$ and $S'$ be the intersections $Q \cap T_pQ$ and $S \cap L$, respectively.
    Then $Q'$ is a quadratic surface which is singular at $p$.

    Assume for contradiction that $Q$ contains a surface~$T$ of \point{1}s.
    Then the secant variety~$\Sec(T)$ is contained in the \locus{2} of $S$\Skern,
    and is thus a component of either $S$ or $Q$.
    This is impossible,
    since both are irreducible.
    Since $Q$ does not contain a surface of \point{1}s,
    $Q'$ contains at most finitely many \point{1}s for a general choice of $p$.
    Moreover, $Q'$ is contained in the \locus{2} of the quartic symmetroid $S'\mkern-3mu$.
    Thus $Q'$ contradicts \cref{prop:max-sing-quadric}. 
\end{proof}

\noindent
The assumption about finitely many \point{1}s
in \cref{prop:max-sing-quadric} is necessary,
as shown below:

\begin{example}
    \label{ex:double-plane}
    The matrix
    \begin{equation*}
        \begin{bmatrix}
            x_0 & x_1 &    x_2    &     0     \\
            x_1 & x_3 &     0     &    x_4    \\
            x_2 &  0  &     0     & x_2 + x_4 \\
             0  & x_4 & x_2 + x_4 &     0
        \end{bmatrix}    
    \end{equation*}
    defines a quartic symmetroid in $\C\P^4$ which has the double plane~$\V\big(x_2 + x_4, x_4^2\big)$ in its \locus{2} and contains the smooth conic section~$\V\big(x_2, x_4, x_0x_3 - x_1^2\big)$ of \point{1}s.

    A quartic symmetroid cannot contain a $3$-space~$H$ of \point{2}s
    with a quadratic surface~$Q \subset H$ of \point{1}s.
    Indeed,
    if that were the case,
    $\Bl(W(H)) \subset \P^3$ would contain a line by \cref{lem:hyperplane-in-base}.
    This is impossible,
    since a line in $\P^3$ is only contained in a conic of \quadric{1}s.
    This example is thus maximal.
\end{example}

\section{Threefolds}
\label[section]{sec:threefolds}

\noindent
Let $S$ be a quartic symmetroid of maximal dimension that is singular along an irreducible quadric~$Q$ of codimension~$1$.
Then $S$ is a threefold and $Q$ is a smooth surface of \point{2}s by \cref{cor:max-smooth}.
In this section,
we describe the singular locus of $S$\Skern,
first in general and then specialise to spectrahedral symmetroids.
The existence of $S$ is shown in \cref{ex:max-smooth-quadric}.

Note first that the set of quadrics associated to $Q$ has a nice description:

\begin{remark}
    \label{rmk:bl-quadric}
    Let $L_1, L_2 \subset \P^3$ be two lines, and let $H_i$ be a plane containing $L_i$.
    The set of pairs $H_1 \cup H_2$ forms a smooth quadratic surface in the $\P^9$ of all quadrics in $\P^3$\Pnkern.
    Any smooth quadratic surface $Q \subset \P^9$ of \quadric{2}s arises this way.
    Indeed, it is straightforward to see that $Q$ contains at most one \quadric{1}.
    Thus by \cref{lem:rank-2-pencil},
    if $P_1$ and $P_2$ are pencils from each of the two rulings on $Q$,
    then $\Bl(P_i)$ consists of a plane~$H_i$ and a line~$L_i$.
    It follows that $L_1 \cup L_2 \subseteq \Bl(Q)$.
\end{remark}

\noindent
Since $Q$ has codimension~$1$ in $S$\Skern,
\cref{rmk:bl-quadric} imposes significant conditions on $W(S)$,
and thus on $S$:

\begin{lemma}
    \label{lem:four-basepoints}
    Let $S \subset \P^4$ be a general quartic symmetroid singular along a smooth quadratic surface~$Q$.
    Then $\Bl(W(S))$ consists of four general points.
\end{lemma}

\begin{proof}
    The linear system~$W(S)$ is spanned by $W(Q)$ and a quadric~$K \notin W(Q)$.
    By \cref{rmk:bl-quadric}, $\Bl(W(Q))$ is the union of two lines, $L_1$ and $L_2$.
    The assertion follows since $K$ intersects $L_1$ and $L_2$ in two points each.
\end{proof}

\begin{proposition}
    \label{prop:two-conics}
    Let $S \subset \P^4$ be a general quartic symmetroid singular along a smooth quadratic surface~$Q$.
    Then $S$ is singular along two additional conic sections.
\end{proposition}

\begin{proof}
    Since $S$ is general,
    the only other singularities are \point{2}s
    outside of $Q$.
    By \cref{lem:four-basepoints},
    $X_{\Bl(W(S))} = X_{\Bl}$ in the notation of \cref{rmk:four-basepoints}.
    We can identify the \point{2}s of $S$
    with the intersection of $W(S)$ with the \locus{2} of $X_{\Bl}$.

    Continuing with the notation of \cref{rmk:four-basepoints},
    consider the intersection $X_{ij} \cap X_k$.
    Since $X_{ij} = X_{kl}$, we may assume that the indices $i$, $j$ and $k$ are distinct.
    The intersection consists of all quadrics $H_{kl} \cup H$\Hkern,
    where $H_{kl}$ is any plane containing $\overline{k, l}$,
    and $H \coloneqq \overline{i, j, l}$.
    Thus $X_{ij} \cap X_k$ is a line.

    The linear system~$W(S)$ is a hyperplane in $X_{\Bl} = \P^5$\Pnkern.
    By \bezout, a generic hyperplane intersects $X_{ij}$ in a conic section and $X_i$ in a line.
    Suppose that the quadratic surface~$Q$ contained in $W(S)$ is $X_{12}$.
    Then $W(S)$ intersects $X_{13}$ and $X_{14}$ in a conic section each.
    Since $X_{12}$ intersects each plane $X_i$ in a line,
    $W(S)$ does not meet $X_i$ outside of $X_{12} = Q$.
\end{proof}

\subsection{Spectrahedra}
\label{sec:threefolds-spectrahedra}

\cref{prop:two-conics} specifies the singularities of $S$
considered as a complex variety.
If $S$ is a real, spectrahedral symmetroid,
then some of those singularities are real
and some of the real singularities lie on the boundary of the spectrahedron.
Here we examine the possible configurations of real singularities.

If $S$ has a nonempty spectrahedron,
then $W(S)$ contains a positive definite quadric.
It follows that $\Bl(W(S))$ consists of complex conjugate pairs, $p_1$, $\overline{p}_1$ and $p_2$, $\overline{p}_2$.
As pointed out in \cite[Lemma~2.4]{HR18},
the real quadrics containing $\overline{p_1, \overline{p}_1}$ and $\overline{p_2, \overline{p}_2}$ are indefinite,
whereas the real quadrics containing $\overline{p_1, p_2}$ and $\overline{\overline{p}_1, \overline{p}_2}$
or $\overline{p_1, \overline{p}_2}$ and $\overline{\overline{p}_1, p_2}$ are semidefinite.
We rename $X_{12}$, $X_{13}$ and $X_{14}$ from \cref{rmk:four-basepoints}
as $X_{i}$, $X_{s1}$ and $X_{s2}$,
indicating whether the real points are indefinite or semidefinite.

\begin{proposition}
    \label{prop:configurations}
    Let $S \subset \P^4$ be a quartic spectrahedral symmetroid singular along a quadratic surface~$Q$.
    Then $S$ is also singular along two conic sections with real points,
    and one of the following is true:
    \begin{enumerate}
        \item
        $Q$ is disjoint from the spectrahedron,
        and both conics lie on the boundary of the spectrahedron.

        \item
        $Q$ and one of the conics lie on the boundary of the spectrahedron,
        while the other conic is disjoint from the spectrahedron.
    \end{enumerate}
\end{proposition}

\begin{proof}
    Choose coordinates such that $p_1 \coloneqq [1 : i : 0 : 0]$ and $p_2 \coloneqq [0 : 0 : 1 : i]$.
    Then $X_{\Bl(W(S))}$ is parametrised by the matrix
    \begin{equation*}
        \begin{bmatrix*}
            x_{00} &   0    & x_{02} & x_{03} \\
              0    & x_{00} & x_{12} & x_{13} \\
            x_{02} & x_{12} & x_{22} &   0    \\
            x_{03} & x_{13} &   0    & x_{22}
        \end{bmatrix*}
        \MXkern,
    \end{equation*}
    and $X_i$ is given by
    \begin{equation}
        \label{eq:Xi}
        \V(x_{00}, x_{22}, x_{02}x_{13} - x_{03}x_{12}).
    \end{equation}
    For the first case, assume that the quadratic surface~$Q$ contained in $W(S)$ is $X_i$.
    Then $W(S)$ is in the pencil of hyperplanes $\V(\delta x_{00} - \lambda x_{22})$.
    If either $\delta = 0$ or $\lambda = 0$,
    then the \locus{2} of $S$ contains two additional planes,
    so we may assume that $W(S)$ equals $\V(x_{00} - \lambda x_{22})$ for $\lambda \neq 0$.
    Hence $W(S)$ is parametrised by
    \begin{equation*}
        \begin{bmatrix*}
            x_{00} &      0 &         x_{02} &         x_{03} \\
                 0 & x_{00} &         x_{12} &         x_{13} \\
            x_{02} & x_{12} & \lambda x_{00} &              0 \\
            x_{03} & x_{13} &              0 & \lambda x_{00}
        \end{bmatrix*}
        \MXkern.
    \end{equation*}
    Thus $W(S)$ intersects $X_{s1}$ and $X_{s2}$ in the conic sections
    \begin{align*}
        &\V\big(x_{02} + x_{13}, x_{03} - x_{12}, \lambda x_{00}^2 - x_{02}^2 - x_{12}^2\big),
        \\
        &\V\big(x_{02} - x_{13}, x_{03} + x_{12}, \lambda x_{00}^2 - x_{02}^2 - x_{12}^2\big),
    \end{align*}
    both of which have real points if and only if $\lambda > 0$.
    Clearly, $W(S)$ contains a positive definite matrix for $\lambda > 0$,
    so $S$ is spectrahedral.
    Recall that a positive semidefinite matrix is the square of a real symmetric matrix.
    It follows that $W(S)$ does not contain a positive definite matrix for $\lambda < 0$,
    so $S$ is not spectrahedral.

    For the second case, assume that the quadratic surface~$Q$ contained in $W(S)$ is $X_{s1}$.
    The surfaces $X_{s1}$ and $X_{s2}$ intersect in the two points
    \begin{equation*}
        \big[\overline{p_1, \overline{p}_1, p_2} \cup \overline{\overline{p}_1, p_1, \overline{p}_2}\big]
        \quad
        \text{and}
        \quad
        \big[\overline{p_1, p_2, \overline{p}_2} \cup \overline{\overline{p}_1, \overline{p}_2, p_2}\big],
    \end{equation*}
    which are both real.
    So any hyperplane containing $X_{s1}$ intersects $X_{s2}$ in a conic section with real points.
    Finally, by \eqref{eq:Xi}, all conic sections on $X_i$ have real points.
\end{proof}

\begin{example}
    \label{ex:max-smooth-quadric}
    The matrix
    \begin{equation*}
        \begin{bmatrix}
            x_0 &  0  & x_1 & x_2 \\
             0  & x_0 & x_3 & x_4 \\
            x_1 & x_3 & x_0 &  0  \\
           x_2 & x_4 &  0  & x_0
        \end{bmatrix}
    \end{equation*}
    defines a quartic spectrahedral symmetroid in $\C\P^4$ 
    which is singular along the smooth quadratic surface~$\V(x_0, x_1x_4 - x_2x_3)$.
    The quadric is disjoint from the spectrahedron.
    
    The matrix
    \begin{equation*}
        \begin{bmatrix}
            x_0 &  0  & x_1 & x_2 \\
             0  & x_0 & x_2 & x_3 \\
            x_1 & x_2 & x_4 &  0  \\
            x_2 & x_3 &  0  & x_4
        \end{bmatrix}
    \end{equation*}
    defines a quartic spectrahedral symmetroid in $\C\P^4$ which is singular along the smooth quadratic surface~$\V\big(x_1 + x_3, x_0x_4 - x_2^2 - x_3^2\big)$.
    The quadric lies on the boundary of the spectrahedron.
\end{example}

\section{Maximality of cyclides}
\label{sec:threefolds-cyclides}

\noindent
In the classical literature,
a real quartic surface singular along a conic section without real points is known as a \emph{cyclide} \cite[Chapter~V]{Jes16}.
Typically, the plane at infinity is chosen to be the plane spanned by the conic. 
Cyclides were originally studied by Dupin \cite{Dup22},
and subsequently by Darboux in a more general setting \cite{Dar73}.
In \cite[Proposition~2.12]{HR18},
it is shown that a so-called \emph{spindle cyclide} with two real nodes
is the only type of cyclide that can occur as a spectrahedral symmetroid.
We extend the definition of a cyclide to higher dimensions:

\begin{definition}
    A \emph{cyclide} is a real, quartic hypersurface in $\C\P^n$ which is singular along a smooth, quadratic $(n - 2)$-fold with no real points.
\end{definition}

\noindent
We know from \cref{cor:max-smooth} that a symmetroidal cyclide is at most $3$-dimensional.
In this section,
we show that such a threefold cannot exist.

Suppose for contradiction that there exists a symmetroid~$S \subset \C\P^4$
that is singular along a quadratic surface~$Q$ with no real points.
Then $Q$ is necessarily smooth
and $\Bl(W(Q))$ contains two lines, $L_1$ and $L_2$,
by \cref{rmk:bl-quadric}.
Since $Q$ contains no real points,
the associated quadric $H_1 \cup H_2$ at any point on $Q$
is not the union of two real or two complex conjugate planes.
This imposes conditions on $L_1$ and $L_2$:

\begin{lemma}
    \label{lem:line-real-plane}
    A line $L \subset \C\P^n$ is contained in a real plane
    if and only if $L$ meets its complex conjugate line~$\overline{L}$.
\end{lemma}

\begin{proof}
    A real line is trivially contained in a real plane.
    If $L$ is not real and meets $\overline{L}$ in a point,
    then $L$ and $\overline{L}$ span a plane which is real.
    If $L$ is contained in a real plane~$H$\Hkern,
    then $\overline{L}$ is also contained in $H$ and therefore meets $L$.
\end{proof}

\begin{lemma}
    \label{lem:lines-conjugate-planes}
    Let $L_1, L_2 \subset \C\P^n$ be two lines.
    Then there exists a plane~$H$ such that $L_1 \subset H$ and $L_2 \subset \overline{H}$
    if and only if $L_1$ meets $\overline{L}_2$.
\end{lemma}

\begin{proof}
    If $L_1 = \overline{L}_2$, then $H$ trivially exits.
    If $L_1 \ne \overline{L}_2$ meet,
    then they span a plane~$H$\Hkern.
    Hence, $L_2$ is contained in the complex conjugate plane~$\overline{H}$\Hkern.
    Conversely, if $L_1 \subset H$ and $L_2 \subset \overline{H}$\Hkern,
    then $\overline{L}_2 \subset H$\Hkern,
    so $L_1$ and $\overline{L}_2$ meet.
\end{proof}

\begin{remark}
    \label{rmk:bl-quadric-no-real}
    Consider $\P^9$ as the set of quadrics in $\P^3$\Pnkern.
    Let $Q \subset \P^9$ be a smooth quadratic surface of \quadric{2}s.
    By \cref{rmk:bl-quadric},
    $\Bl(Q)$ consists of two lines, $L_1$ and $L_2$.
    Suppose that $Q$ contains no real points.
    A \quadric{2} is real if and only if it is the union of two real or two complex conjugate planes.
    By \cref{lem:line-real-plane,lem:lines-conjugate-planes},
    $Q$ contains no real points if and only if $L_1 \cap \overline{L}_2 = \varnothing$ and at least one of the intersections~$L_1 \cap \overline{L}_1$ or $L_2 \cap \overline{L}_2$ is empty.
    In particular,
    $L_1$ and $L_2$ are not complex conjugates and at most one of them is real.
\end{remark}

\noindent
All symmetroids that are cyclides are covered by the original $2$-dimensional definition of a cyclide:

\begin{proposition}
    \label{prop:cyclide}
    Let $S \subset \C\P^n$ be a symmetroidal cyclide.
    Then $S$ is a surface.
\end{proposition}

\begin{proof}
    Assume for contradiction that $S$ is a threefold.
    Let $Q \subseteq \Sing(S)$ be the quadratic surface with no real points.
    Since $S$ is real, so is $Q$.
    Hence $W(Q)$ spans a real web~$W \subset \P^9$ of quadrics.
    Because $W$ is real and linear,
    it is spanned by real quadrics.
    Thus $\Bl(W)$ is real.
    Since $W$ is spanned by $W(Q)$,
    we have $\Bl(W) = \Bl(W(Q))$.
    Hence the two lines in $\Bl(W(Q))$ are either both real or complex conjugates.
    This contradicts \cref{rmk:bl-quadric-no-real}.
\end{proof}

\chapter{Maximality of reducible quadrics}
\label[section]{sec:max-red}

\cref{sec:max-irred} finds the maximal dimension
of irreducible quadratic components in the singular locus of a quartic symmetroid.
It is natural to ask whether this is the maximal dimension among all double quadrics.
We show here that there exists a quartic symmetroid~$S \subset \P^5$
that is singular along a reducible, quadratic threefold~$Q$
and that $S$ is maximal.

We argue that the points in $Q$ are generically \point{3}s:
Let $H$ be a $3$-space in $Q$.
Assume for contradiction that $H$ is contained in the \locus{2} of $S$\Skern.
By \cref{ex:double-plane},
$H$ cannot contain infinitely many \point{1}s.
Because $H$ contains finitely many \point{1}s,
\cref{lem:hyperplane-in-base} implies that $\Bl(W(H))$
contains a plane.
Since $W(S)$ is generated by $W(H)$ and two quadrics,
it follows that $\Bl(W(S))$ consists of four coplanar points,
so $X_{\Bl(W(S))}$ is $5$-dimensional.
Hence $W(S) = X_{\Bl(W(S))}$
and $S$ is the discriminant of $X_{\Bl(W(S))}$.
In the notation of \cref{rmk:four-basepoints},
the only threefold that the discriminant of $X_{\Bl(W(S))}$ is singular along is $X$\Xkern,
which is not a quadric.

\begin{proposition}
    \label{prop:max-union}
    Let $S \subset \P^n$ be an irreducible, quartic symmetroid.
    Assume that $Q$ is an $(n - 2)$-dimensional quadric in the singular locus of $S$.
    Then $n \leqslant 5$.
\end{proposition}

\begin{proof}
    Consider $\P^9$ as the set of all quadrics in $\P^3$\Pnkern.
    Our aim is to find a maximal linear subspace $W \subset \P^9$
    such that the discriminant of $W$ is singular along a quadric
    of dimension $\dim(W) - 2$.

    Assume first that $Q \coloneqq H_1 \cup H_2$ is the union of
    two linear subspaces of \point{3}s.
    \cref{lem:common-singularity}
    implies that the quadrics associated to $H_i$
    share a common singular point~$p_i$,
    which is a \basepoint for $W(S) \subset \P^9$\Pnkern.
    Hence $W(S) \subset X_{\{p_1, p_2\}} = \P^7$\Pnkern.
    Let $Y_i \subset X_{\{p_1, p_2\}}$ be the $4$-space
    of quadrics that are singular at $p_i$.
    We identify $H_i$ with $W(S) \cap Y_i$.
    The set of quadrics that are singular at both $p_1$ and $p_2$ forms a double plane.
    Hence $Y_1$ and $Y_2$ intersect in a plane,
    so the span $\overline{Y_1 \cup Y_2}$ of $Y_1 \cup Y_2$ is a $6$-space.
    All quadrics in $\overline{Y_1 \cup Y_2}$ contain
    the line spanned by $p_1$ and $p_2$,
    so the discriminant of $\overline{Y_1 \cup Y_2}$ is reducible by \cref{lem:base-curve}.
    Thus $W \neq \overline{Y_1 \cup Y_2}$
    and we conclude that $n \leqslant 5$.

    By \cref{lem:codim-2-singular-subspace},
    if $Q$ is a double linear subspace,
    the situation is analogous.
    Indeed,
    $\Bl(W(S))$ contains a scheme~$2p$ of length~$2$ with support in one point~$p$.
    Let $Y \subset X_{2p} = \P^7$ be the $4$-space
    of quadrics that are singular at $p$
    and contain the line spanned by $2p$.
    Let $D$ be the discriminant of $X_{2p}$.
    The singular locus of $D$ contains a double subspace~$2Y$
    whose underlying, reduced space is $Y$\Ykern.
    The fourfold~$2Y$ spans a $6$-space.
    By letting $2Y$ play the role of $Y_1 \cup Y_2$,
    we get the same conclusion as in the previous paragraph.
\end{proof}

\section{Fourfolds}
\label{sec:fourfolds}

Let $S$ be a quartic symmetroid of maximal dimension that is singular
along a reducible quadric of codimension~$1$ in $S$\Skern.
Then $S$ is at most $4$-dimensional by \cref{prop:max-union}.
In this section,
we describe the rest of the singular locus of such a fourfold.
The existence is demonstrated in \cref{ex:double-P3,ex:two-P3s}.

\begin{proposition}
    \label{prop:additional-quad-surf}
    Let $S \subset \P^5$ be a general quartic symmetroid
    singular along two $3$-spaces, $H_1$ and $H_2$.
    Then $S$ is singular along an additional smooth quadratic surface.
\end{proposition}

\begin{proof}
    Since $S$ is general,
    the only other singularities are \point{2}s
    outside of $H_1$ and $H_2$.
    We can identify the \point{2}s of $S$
    with the intersection of $W(S)$ with the \locus{2} of $X_{\Bl(W(S))}$.

    \cref{lem:common-singularity} tells us that
    the quadrics associated to $H_i$, for $i = 1, 2$, 
    have a common singularity~$p_i$.
    By applying \cref{lem:codim-2-singular-subspace} to $H_1$ and $H_2$,
    we get that $\Bl(W(S))$
    consists of two schemes $2p_i$
    of length~$2$ with support in $p_i$,
    for $i = 1, 2$.
    Since $\Bl(W(S))$ is of length~$4$,
    $X_{\Bl(W(S))}$ is a degeneration of $X_{\Bl}$ in \cref{rmk:four-basepoints}.
    Moreover, $W(S) = X_{\Bl(W(S))}$.

    For a \quadric{2}~$H \cup H' \subset \P^3$ to contain the scheme~$2p_1$,
    there are two possibilities:
    Either the line spanned by $2p_1$ is contained in the plane~$H$\Hkern,
    or both $H$ and $H'$ contain $p_1$.
    In the latter case,
    $H \cup H'$ is singular at $p_1$,
    so $[H \cup H']$ is contained in $H_1$,
    the set of all quadrics in $W(S)$ singular at $p_1$.
    Hence only one of the quadratic surfaces described in \cref{rmk:four-basepoints}
    is not contained in $H_1$ or $H_2$:
    The set of all unions $H \cup H'\mkern-5mu$,
    where $H$ contains the line spanned by $2p_1$
    and $H'$ contains the line spanned by $2p_2$.
\end{proof}

\begin{proposition}
    \label{prop:sing-loc-double-P3}
    Let $S \subset \P^5$ be a general quartic symmetroid
    with a double $3$-space~$2H$ in $\Sing(S)$.
    Then $S$ is singular along an additional plane.
\end{proposition}

\begin{proof}
    Since $S$ is general,
    the only other singularities are \point{2}s
    outside of $2H$\Hkern.
    We can identify the \point{2}s of $S$
    with the intersection of $W(S)$ with the \locus{2} of $X_{\Bl(W(S))}$.

    Recall from the discussion at the start of \cref{sec:max-red}
    that the points in $H$ are generically \point{3}s.
    Thus \cref{lem:common-singularity} implies that
    the quadrics associated to $H$ are singular at a point~$p$.
    By \cref{lem:codim-2-singular-subspace},
    $\Bl(W(S))$ contains a scheme~$2p$ of length~$2$ with support in $p$,
    so $X_{\Bl(W(S))} \subseteq X_{2p}$.
    For a \quadric{2}~$H_1 \cup H_2 \subset \P^3$ to contain $2p$,
    there are two possibilities:
    Either the line~$L$ spanned by $2p$ is contained in the plane~$H_1$,
    or both $H_1$ and $H_2$ contain $p$.
    In the latter case,
    $H_1 \cup H_2$ is singular at $p$,
    so $[H_1 \cup H_2]$ is contained in $H$\Hkern.
    Hence if $H_1 \cup H_2$ is the associated quadric
    to a \point{2} outside of $H$\Hkern,
    then $H_1$ contains $L$.

    The set of pairs $H_1 \cup H_2$,
    where $H_1$ contains $L$,
    forms a fourfold~$Z_L$ of degree~$4$ in $X_{2p} = \P^7$\Pnkern.
    The singular locus of $Z_L$ is the plane of pairs $H_1 \cup H_2$,
    where both $H_1$ and $H_2$ contain $L$.
    Let $Y \subset X_{2p}$ be the $4$-space of quadrics
    that are singular at $p$ and contains $L$,
    as in the proof of \cref{prop:max-union}.
    We identify $H$ as $Y \cap W(S)$.
    The intersection $Z_L \cap Y$ is the set of all unions $H_1 \cup H_2$,
    where $H_1$ contains $L$ and $H_2$ contains $p$.
    This is a divisor on $Z_L$ that contains $\Sing(Z_L)$;
    it is a threefold of degree~$3$.
    The linear system $W(S) = \P^5$ intersects $Z_L$ in a quartic surface.
    Since $W(S)$ intersects $Y$ in the $3$-space~$H$\Hkern,
    it meets $Z_L \cap Y$ in a cubic surface.
    Hence $W(S)$ intersects $Z_L$ in a plane outside of $H$\Hkern.
    This proves the claim.
\end{proof}

\begin{example}
    \label{ex:double-P3}
    The matrix
    \begin{equation*}
        \begin{bmatrix}
            x_0 + x_1 & x_0 + x_2 & x_3 & x_2 \\
            x_0 + x_2 & x_0 - x_1 & x_3 & x_2 \\
               x_3    &    x_3    & x_4 & x_5 \\
               x_2    &    x_2    & x_5 &  0
        \end{bmatrix}
    \end{equation*}
    defines a quartic symmetroid $S \subset \C\P^5$
    whose Jacobian ideal is contained in the ideal $\ip[\big]{x^2_2, x_5}$.
    In other words,
    the singular locus of $S$ contains a double $\P^3$\Pnkern.
    The symmetroid is also singular along the plane~$H_1 \coloneqq \V(x_0, x_1, x_2)$.

    The \locus{2} of $S$ consists of $H_1$,
    three times the plane~$H_2 \coloneqq \V(x_2, x_4, x_5)$
    and six times the plane~$H_3 \coloneqq \V(x_1, x_2, x_5)$.
    Note that $H_3$ equals $\Sing(Z_L)$
    in the proof of \cref{prop:sing-loc-double-P3}.
    It contains the conic~$\V\big(x_1, x_2, x_5, x_0x_4 - x_3^2\big)$ of \point{1}s.
\end{example}

\begin{remark}
    If $S$ is a quartic symmetroid singular along
    the union of two distinct linear subspaces
    with codimension~$1$ in $S$\Skern,
    then \cref{lem:codim-2-singular-subspace} implies that
    $\Bl(W(S))$ is a scheme of length~$4$.
    Hence $X_{\Bl(W(S))} = \P^5$\Pnkern.
    If $S$ is a fourfold,
    it follows that $W(S) = X_{\Bl(W(S))}$.
    Thus $S$ is the discriminant of $X_{\Bl(W(S))}$,
    so it is uniquely determined by $\Bl(W(S))$.

    However, if $S \subset \P^5$ is a quartic symmetroid
    with a double $3$-space in its singular locus,
    then $\Bl(W(S))$ is too small to be unique for $S$\Skern.
    In particular,
    we compute that $\Bl(W(S))$ is a scheme of length~$3$
    for $S$ in \cref{ex:double-P3}.
\end{remark}

\subsection{Spectrahedra}

Here we examine the real singularities
in the case of spectrahedral symmetroids.
If $S$ has a nonempty spectrahedron,
then the \basepoints of $W(S)$ appear in complex conjugate pairs.
It follows that the $3$-spaces in the singular locus of $S$ are also complex conjugates,
since they correspond to quadrics in $W(S)$ that are singular at one of the \basepoints.
This means that $\Sing(S)$ cannot contain a double $3$-space,
only two distinct $3$-spaces.

\begin{proposition}
    Let $S \subset \P^5$ be a quartic spectrahedral symmetroid singular along two
    complex conjugate $3$-spaces, $H$ and $\overline{H}$\Hkern.
    The plane $H \cap \overline{H}$ of real points intersects the spectrahedron
    in a region bounded by a conic section of \point{1}s.

    The symmetroid~$S$ is also singular along a smooth quadratic surface~$Q$ whose real points lie on the spectrahedron.
\end{proposition}

\begin{proof}
    The plane $H \cap \overline{H}$ corresponds to quadrics that are singular
    along the real line~$L$ spanned by the two reduced \basepoints of $W(S)$,
    $p$ and $\overline{p}$.
    These include all real double planes containing $L$,
    the set of which corresponds to
    a conic section~$C$ of real \point{1}s in $H \cap \overline{H}$\Hkern.
    Let $l_1$, $l_2$ be some linear forms such that $\big[\V\big(l^2_1\big)\big]$ and $\big[\V\big(l^2_2\big)\big]$ are two points in $C$\Ckern.
    Consider the line spanned by the points:
    The quadric $\V\big(al^2_1 + bl^2_2\big)$ is semidefinite
    if and only if the constants~$a$ and $b$ have the same sign.
    Hence the intersection between $H \cap \overline{H}$ and the spectrahedron
    is bounded by $C$\Ckern.

    From the proof of \cref{prop:additional-quad-surf},
    we know that the points in $Q$ correspond to pairs of planes~$H_1 \cup H_2$,
    where $H_1$ contains the line spanned by $2p$
    and $H_2$ contains the line spanned by $2\overline{p}$.
    If $H_1 \cup H_2$ is real,
    then $H_1 = \overline{H}_2$,
    hence $H_1 \cup H_2$ is semidefinite.
\end{proof}

\begin{example}
    \label{ex:two-P3s}
    The matrix
    \begin{equation*}
        \begin{bmatrix*}[r]
            x_0 &  x_1 &  x_2      & x_3      \\
            x_1 &  x_4 & -x_3      & x_2      \\
            x_2 & -x_3 &  x_5      &  0\,\,\, \\
            x_3 &  x_2 &   0\,\,\, & x_5
        \end{bmatrix*}
    \end{equation*}
    defines a quartic spectrahedral symmetroid in $\C\P^5$ which is singular along
    the two complex conjugate $3$-spaces $\V(x_2 \pm ix_3, x_5)$
    and the quadratic surface $\V\big(x_0 - x_4, x_1, x_2^2 + x_3^2 - x_4x_5\big)$.
\end{example}

\printbibliography

\paragraph{Author's address:}

Martin Hels\o,
University of Oslo,
Postboks 1053 Blindern,
0316 Oslo,
Norway,
\href{mailto:martibhe@math.uio.no}{\nolinkurl{martibhe@math.uio.no}}

\end{document}